\documentclass{amsart}

\usepackage{amsmath,amsthm,amssymb}
\usepackage[all,cmtip]{xy}
\usepackage{url}

\theoremstyle{plain}
\newtheorem{theorem}{Theorem}[section]
\newtheorem{lemma}[theorem]{Lemma}
\newtheorem{corollary}[theorem]{Corollary}
\newtheorem{proposition}[theorem]{Proposition}

\theoremstyle{definition}
\newtheorem{definition}[theorem]{Definition}
\newtheorem{remark}[theorem]{Remark}
\newtheorem{example}[theorem]{Example}

\newcommand{\vphi}{\varphi}

\newcommand{\A}[1]{\forall #1\,}

\newcommand{\up}{\uparrow}
\newcommand{\down}{\downarrow}
\newcommand{\updown}{\uparrow\mathrel{\mspace{-2mu}}\downarrow}

\begin{document}

\title{Actions arising from intersection and union}

\author{Alex Kruckman}
\author{Lawrence Valby}

\subjclass[2010]{Primary: 03C05; Secondary: 08C15}

\keywords{actions, biactions, axiomatization}

\begin{abstract}
An action is a pair of sets, $C$ and $S$, and a function $f\colon C\times S \to C$. Rothschild and Yalcin gave a simple axiomatic characterization of those actions arising from set intersection, i.e.\ for which the elements of $C$ and $S$ can be identified with sets in such a way that elements of $S$ act on elements of $C$ by intersection. We introduce and axiomatically characterize two natural classes of actions which arise from set intersection and union. In the first class, the $\updown$-actions, each element of $S$ is identified with a pair of sets $(s^\down,s^\up)$, which act on a set $c$ by intersection with $s^\down$ and union with $s^\up$. In the second class, the $\updown$-biactions, each element of $S$ is labeled as an intersection or a union, and acts accordingly on $C$. We give intuitive examples of these actions, one involving conversations and another a university's changing student body. The examples give some motivation for considering these actions, and also help give intuitive readings of the axioms. The class of $\updown$-actions is closely related to a class of single-sorted algebras, which was previously treated by Margolis et al., albeit in another guise (hyperplane arrangements), and we note this connection. Along the way, we make some useful, though very general, observations about axiomatization and representation problems for classes of algebras.
\end{abstract}

\maketitle

\section{Introduction}\label{intro}

An \emph{action} (of $S$ on $C$ on the right) is a pair of sets, $C$ and $S$, and a function $f\colon C\times S \to C$. We denote by $S^*$ the set of words in $S$ (i.e.\ finite sequences of elements of $S$, including the empty sequence). For brevity, we write $f(c,s)$ as $cs$, so that given $c\in C$ and $w\in S^*$, $cw$ is an element of $C$.

One intuitive interpretation of actions has been given by philosophers studying conversational dynamics (as in \cite{Yalcin}, for example). Given an action $(C,S)$, we can think of $C$ as the states that a conversation can have, and $S$ as the sentences which, when said, change the state. A natural class of concrete models can be described by taking both the states $c\in C$ and the sentences $s\in S$ to be sets of possible worlds. Then saying $s$ in state $c$ corresponds to cutting down the set of possible worlds by intersection $c\cap s$.

With this motivation, Rothschild and Yalcin pointed out in \cite{Yalcin} that the actions which can be expressed using set intersection in this way are exactly the idempotent, commutative actions. In detail, an action $(C,S)$ is called \emph{idempotent} when $css=cs$ and \emph{commutative} when $cs_1s_2=cs_2s_1$. When the elements of $C$ and $S$ can be identified with subsets of some set in such a way that $cs=c\cap s$, then we say that the action can be \emph{expressed using set intersection}.

This observation of Rothschild and Yalcin, which we restate below, is a close relative of the representation theorem for semilattices: every semilattice is a subalgebra of the semilattice of subsets of $A$ under intersection, for some set $A$.

\begin{theorem}[Rothschild and Yalcin]\label{intersection} 
An action can be expressed using set intersection if and only if it is idempotent and commutative.
\end{theorem} 
\begin{proof}
It is easy to check that any action expressible using set intersection is idempotent and commutative. To see the other direction, one may identify an element $c\in C$ with $O(c)=\{cw\mid w\in S^*\}$, the orbit of $c$, and identify $s\in S$ with $F(s)=\{c\mid cs=c\}\cup\{s\}$, the fixed points of $s$ together with the tag ``$s$" to ensure $F$ is 1-1. From idempotence and commutativity it follows that if $w\in S^*$ and $dw=c$, then every element of $O(c)$ is fixed by $w$. We claim that $O$ is 1-1. If $O(c)=O(d)$ then $d\in O(c)$ and $c\in O(d)$ and so there is $w\in S^*$ with $dw=c$. By our earlier observation it follows that $d$ is fixed by $w$, so $d=dw=c$. Finally we can check that $O(cs)=O(c)\cap F(s)$. Let $csw\in O(cs)$ (where $w\in S^*$). Of course $csw\in O(c)$. Further $csws=cssw=csw$, so $csw\in F(s)$. Now let $cw\in O(c)\cap F(s)$ (where again $w\in S^*$). Then $cws=cw$, and so $csw=cw$. So $cw\in O(cs)$.
\end{proof}

Seeing that we obtain such a tidy axiomatization when looking at intersection, a natural question arises: What happens if we also throw union into the mix? From the conversational dynamics perspective described above, in the purely intersective case, sentences can only rule out possibilities. Allowing union could capture situations in which some sentences rule out possibilities, while others rule possibilities back in.

We address this question in two ways. First, in Section~\ref{actions}, we consider actions in which each element of $S$ acts by both intersection and union. We say an action $(C,S)$ is a \emph{$\updown$-action} if each element of $C$ can be identified with a set, and each element of $S$ can be identified with a pair of sets $(s^\down,s^\up)$, such that $s^\up\subseteq s^\down$, in such a way that the action of $s$ on $c$ is given by $(c\cap s^\down) \cup s^\up$. 

An alternative way of adding in union is to label each element of $S$ as an intersection element or a union element. In this setup, sentences can no longer rule out and rule in possibilities simultaneously; instead, each sentence can only do one or the other. In Section~\ref{biactions}, we introduce the class of \emph{$\updown$-biactions}, so named because they are $3$-sorted algebras $(C,S^\down,S^\up)$ with an action of $S^\down$ on $C$ by intersection and an action of $S^\up$ on $C$ by union.

Surprisingly, both of these cases are significantly more complicated than the case of actions which can be expressed using set intersection. The classes of $\updown$-actions and $\updown$-biactions do not admit equational axiomatizations; however, each class is a quasivariety, axiomatized by finitely many equational axioms (which give the equational theory of the class --- see Propositions~\ref{pmequations} and~\ref{biequations}) together with a single infinite Horn clause schema. The axioms will be explained later (in an intuitive way in Section~\ref{intuition} and in a mathematical way in later sections), but we will write them down here for reference.

$\updown$-actions are axiomatized by idempotence (I), previous redundance (PR), and the strong links axioms (SL). Below, $c$, $d$, and the $a_i$ are variables of sort $C$, $s$ and $t$ are variables of sort $S$, and the $w_i$ are words of sort $S^*$ (arbitrary sequences of variables from $S$).
\begin{align*}
\text{(I)} \hspace{.25in} & css = cs\\
\text{(PR)} \hspace{.25in} & csts = cts\\
\text{(SL)} \hspace{.25in}  & \Bigg(\left(\bigwedge_{i=1}^n cw_i=dw_i\right)\wedge c = a_0 \wedge d=a_n \wedge\\
& \left(\bigwedge_{i=1}^n a_{i-1}w_i=a_{i-1}\wedge a_{i}w_i=a_{i}\right)\Bigg)\rightarrow (c=d)
\end{align*}

$\updown$-biactions are axiomatized by idempotence (I), previous redundance (PR), commutativity (C) in $S^\up$ and $S^\down$, and the subset axioms (S). Below, $c$, $d$, and $e$ are variables of sort $C$, $s$ and $u$ are variables of sort $S^\down$, $t$ and $v$ are variables of sort $S^\up$, and $w$ is a word of sort $(S^\down\cup S^\up)^*$ (an arbitrary sequence of variables from $S^\down$ and $S^\up$).
\begin{align*}
\text{(I)} \hspace{.25in} & css = cs \hspace{.25in} ctt = ct\\
\text{(PR)} \hspace{.25in} & csts = cts \hspace{.25in} ctst = cst\\
\text{(C)} \hspace{.25in} & csu = cus \hspace{.25in} ctv = cvt\\
\text{(S)} \hspace{.25in} & (csw = dsw \land ctw = dtw \land esw = etw)\rightarrow (cw = dw)
\end{align*}

In Section~\ref{intuition}, we give two examples of how $\updown$-actions (and $\updown$-biactions) may arise, and we use these examples to give an intuitive reading of the axioms just stated. The first example has to do with a university's changing student body. The second has to do with conversations, but we take something of a different approach from the possible worlds framework mentioned above. The section thus serves two functions. On the one hand it gives an intuitive perspective on the mathematical structures under discussion in this paper, and on the other hand it supplies some motivation for considering these structures in the first place.

Indeed, the conversation example was our original motivation for studying $\updown$-actions and $\updown$-biactions. In future work, it would be interesting to study these actions as models for conversation from the point of view of formal semantics and the philosophy of language, and to compare with existing approaches in the literature (for example inquisitive and attentive semantics in \cite{Ciardelli} and \cite{Roelofsen}). However, in this paper, we focus mainly on the axiomatization problem for the actions in question. This is a natural mathematical question which is only indirectly motivated by the examples in Section~\ref{intuition}; nevertheless, our solutions to the axiomatization problems do increase our understanding of $\updown$-actions and $\updown$-biactions in ways which could be useful in future work.

In Section~\ref{generalities}, we review some general results on axiomatization and representation problems, which are needed for the rest of the paper. In particular, the classes of $\updown$-actions and $\updown$-biactions can both be described using a certain operation which takes as input a set $X$ and outputs an algebra $F(X)$, the ``full" $\updown$-action or  $\updown$-biaction on $X$. Then the class of algebras in question is the class of subalgebras of full algebras. We observe that whenever an operation $F$ from sets to algebras turns disjoint unions into products, the quasivariety generated by algebras of the form $F(X)$ is in fact generated by the single algebra $F(1)$, and draw some useful conclusions.

Sections~\ref{actions} and~\ref{biactions} are devoted to the classes of $\updown$-actions and $\updown$-biactions, respectively. The main axiomatization results are Theorem~\ref{main theorem} and Theorem~\ref{biaction theorem}.

The class of $\updown$-actions is closely related to a certain class of single-sorted algebras we call set bands, which we discuss in Section~\ref{set bands}. It turns out that the class of set bands is exactly the quasivariety generated by a certain $3$-element semigroup. This quasivariety was studied and axiomatized in \cite{Margolis}, but with the motivation coming from hyperplane arrangements. The connection to $\updown$-actions provides an additional motivation for studying this quasivariety.

Our solution to the $\updown$-action axiomatization problem was obtained after observing the connection with the single-sorted set bands axiomatization problem, and adapting the solution of Margolis et al.\ in \cite{Margolis} to the action case. The argument in the case of $\updown$-biactions is different than in the case of $\updown$-actions, but it shares the same basic structure.

We would like to acknowledge George Bergman, Richard Lawrence, Tom Scanlon, and Seth Yalcin for helpful suggestions and conversations. The final publication is available at Springer via \url{http://dx.doi.org/10.1007/s10849-016-9240-0}.

\section{Intuitive examples}\label{intuition}

In this section we present two intuitive sources of $\updown$-actions (and $\updown$-biactions). Our goal here is to motivate interest in these actions, and also to give an intuitive explanation of the axioms.  

First, let's consider the student body of a university. We conceive of this simply as a set of students. Based on the behavior of the appropriate university official, the student body will change. Sometimes the official will do something that adds students to the student body, sometimes the official will do something that removes students from the student body, and sometimes the official may do something that both removes and adds students. In such a situation we obtain an $\updown$-action $C\times S\to C$. The elements of $C$ are sets of people, and the elements of $S$ are pairs $s=(s^\down,s^\up)$ of sets of people. Given a fixed $s$ (let's call it an ``act" of the official), the effect of the function $c\mapsto cs$ is to remove from the student body all people not in $s^\down$, and add to the student body all people in $s^\up$. In the special case the official at a given time must either just add or just remove people, we obtain an $\updown$-biaction.

It's reasonable to be interested in the algebraic behavior of this student body action, without wanting to think specifically about the sets involved. For example, observe that if the official performs the same act $s$ twice in a row, this has the same effect as just performing it once. This is the first of our axioms for $\updown$-actions. In detail, idempotence states that for all $c$ in $C$ and all $s$ in $S$ we have $css=cs$. Of course, this student body action is not commutative. If $s$ adds John Doe and $t$ removes John Doe, then the order in which $s$ and $t$ are performed obviously matters. On the other hand, what if the official performs $s$, then performs $t$, and then performs $s$ once again? Certainly the second performance of $s$ is important because, e.g., $t$ may remove some student that $s$ adds. However, the first performance of $s$ may be omitted without changing the net result. In detail, whoever the first $s$ would add or remove will still be added or removed by the second $s$. This is our second axiom for $\updown$-actions. Previous redundance states that for all $c$ in $C$ and all $s$ and $t$ in $S$ we have $csts=cts$.

These two axioms (idempotence and previous redundance) characterize the equational theory of $\updown$-actions. That is, every other equation which holds universally in $\updown$-actions follows from these equations. For example, $csttus = ctus$, since $csttus = cstus = cstsus = ctsus = ctus$. Intuitively, any sequence of acts that the official performs is equivalent to the one where the official only performs the last instance of each distinct act in the sequence (in the inherited order), and further simplification isn't possible in general. The strong links axiom schema is needed to fully axiomatize $\updown$-actions, i.e.\ to be able to say that if an abstract action satisfies the axioms, then it is (isomorphic to) an $\updown$-action. It too may be given an intuitive reading, but let us give this in the context of our next example.

Consider a ``conversation" where we assume there is a shared conversational state that is successively changed by participants saying sentences. We conceive of the conversational state simply as a set of possible facts. For example, perhaps ``It is raining" is in the conversational state. When a possible fact is in the state, the participants are actively considering it as possible. When a participant says a sentence, some possible facts may be added to the state, some possible facts may be removed from the state, or both of these things may happen at once. We thus have an $\updown$-action $C\times S\to C$. The elements of $C$ are the shared conversational states (i.e.\ sets of possible facts), and the elements of $S$ are the sentences, which are formally pairs $s=(s^\down,s^\up)$ of sets of possible facts. Given a fixed sentence $s$, the effect of the function $c\mapsto cs$ is to remove from the state all possible facts not in $s^\down$, and add to the state all possible facts in $s^\up$. In the special case every sentence just adds or just removes possible facts, we obtain an $\updown$-biaction.

Idempotence and previous redundance may be given an intuitive reading here as well, but let's focus now on the strong links axiom schema. We first make a couple of preliminary observations.

When a participant says $s$, certain fixed possible facts are removed or added. Let us call the possible facts that may be removed or added by $s$ ``the material relevant to $s$". Now for our first observation. Suppose that $cs=ds$. That is, although $c$ and $d$ may be different conversational states, nevertheless they become the same state after saying $s$ at each of them. The observation is that $c$ and $d$ must agree on material not relevant to $s$. That is, given some possible fact that is neither removed nor added by $s$, $c$ and $d$ must either both contain that possible fact or both not contain that possible fact. 

Now for the second observation. Suppose that $cs=c$ and $ds=d$. That is, when the state is either $c$ or $d$, and someone says $s$, the state remains as it is. The observation is that $c$ and $d$ must agree on material relevant to $s$. To see this, observe that any possible fact removed by $s$ is in neither $c$ nor $d$, and any possible fact added by $s$ is in both $c$ and $d$. So in fact $c$ and $d$ not only agree on material relevant to $s$, they agree on material relevant to $s$ in the particular way prescribed by $s$.

Armed with these two observations, we are now in a position to give an intuitive explanation of the strong links axiom schema (SL). Let us consider a special case that already illustrates the ideas involved. In the notation of Section~\ref{intro}, the special case we consider is where $n=2$ and the sequences of variables $w_1$ and $w_2$ are just individual variables $s_1$ and $s_2$. Specifically, the axiom states that if you have states $a_0,a_1,a_2$ and sentences $s_1,s_2$ with $a_0s_i=a_2s_i$ and $a_{i-1}s_i=a_{i-1}$ and $a_is_i=a_i$ for $i=1,2$, then $a_0=a_2$. Using our first observation above, $a_0s_i=a_2s_i$ tells us that the states $a_0$ and $a_2$ agree on the material not relevant to $s_i$, for $i=1$ and $i=2$. To show that $a_0$ and $a_2$ agree everywhere, it remains to show that they agree on material relevant both to $s_1$ and $s_2$. Since $a_0s_1=a_0$ and $a_1s_1=a_1$, our second observation above tells us that $a_0$ and $a_1$ agree on material relevant to $s_1$. Similarly, as $a_1s_2=a_1$ and $a_2s_2=a_2$, we get that $a_1$ and $a_2$ agree on material relevant to $s_2$. Thus, $a_0$ and $a_2$ must agree also on material relevant both to $s_1$ and $s_2$, and so they agree on all material and should be the same state. 

A similar intuitive reading may be given of the other strong links axioms (e.g., in the cases $n > 2$) and of the subset axioms for $\updown$-biactions, but we will omit these for the sake of brevity.

Another way to treat conversations would be to conceive of the shared conversational state as a set of possible worlds, rather than a set of possible facts. Indeed, this point of view is apparent in the work of Yalcin and Rothschild \cite{Yalcin} discussed in Section~\ref{intro}. On this view the conversational state contains the possible worlds that the participants jointly consider possible at that point. We could also take this view, and the intuitive explanation of the axioms would still work, but the kind of behavior we might want sentences to have would not be attainable in certain intuitive situations. We now give a concrete example illustrating how the possible worlds view can fall short in the context of $\updown$-actions. Along the way we acquire some additional motivation for both the possible facts conception of the conversational state and $\updown$-actions in general.

Consider a situation where there are 4 possible worlds $W=\{00,01,10,11\}$. Each possible world makes the determination whether it is raining or not, and whether the cat is hungry or not. For example, 10 is the possible world where it is raining (because there is a ``1" in the first coordinate) and the cat is not hungry (because there is a ``0" in the second coordinate). Now assume that there are 6 sentences which may be said: 
\begin{align*}
s_1&=\text{``It may or may not be raining"}\\
s_2&=\text{``The cat may or may not be hungry"}\\
t_1&=\text{``It is raining"}\\
u_1&=\text{``It is not raining"}\\
t_2&=\text{``The cat is hungry"}\\
u_2&=\text{``The cat is not hungry"}
\end{align*}
Under the possible worlds view, we have $C\subseteq\mathcal{P}(W)$, i.e.\ a state $c\in C$ is a set of possible worlds. Let's focus attention on conversations that start from ignorance. That is, the conversations start in state $W$, and so $C$ is taken to be the collection of states reachable from $W$ by saying a sequence of sentences. 

Now we stipulate how the sentences act on the states. First, we stipulate that each of the sentences $t_i,u_i$ acts by intersecting the current state with the appropriate fixed set. For example, $ct_1=c\cap\{10,11\}$ --- intuitively this makes sense because saying $t_1$ should remove the worlds where it is not raining. It's not as obvious how we should assume $s_1$ and $s_2$ act. But one intuitive assumption to make in particular is that $\{10,11\}s_1=W$ and $\{11\}s_1=\{01,11\}$. Intuitively, if we think it's raining, and someone says it may or may not be raining, then we don't think it's raining anymore. However, this sentence doesn't change our views about whether that cat is hungry. More generally, one mathematically natural definition for $s_1$ and $s_2$ extending this particular stipulation is to view them as cylindrifications of the first and second coordinates respectively. E.g., $cs_1$ is the state obtained from $c$ by adding to it all the possible worlds that differ in just the first coordinate from a possible world already in $c$.

Certainly the action above is not directly presented as an $\updown$-action. Indeed, the cylindrifications patently add possible worlds in a way that depends on the current state. But the question remains whether this action is algebraically an $\updown$-action.

The possible worlds setup is very sensitive to contradiction. For example, if someone says it's raining and someone says it's not raining, then we are led to the state $\emptyset$, losing any information we may have had about whether the cat is hungry. It's because of this sensitivity to contradiction that the action above is not an $\updown$-action. E.g., $Wt_1s_1u_1s_1=W\neq \emptyset=Wt_1u_1s_1$, violating previous redundance. 

However, the action is essentially an $\updown$-action if we purposefully avoid contradiction. To see this, we move to the possible facts point of view. Let $W'=\{R,\neg R,H, \neg H\}$ be the collection of possible facts for our scenario (e.g.\ ``$\neg R$" is the possible fact that it is not raining). We define an $\updown$-action as follows: $s_1$ adds $R$ and $\neg R$, $s_2$ adds $H$ and $\neg H$, $t_1$ removes $\neg R$, $t_2$ removes $\neg H$, $u_1$ removes $R$, and $u_2$ removes $H$. This stipulation matches the behavior we might intuitively expect based on the English glosses of the sentences. The algebraic behavior of this action is equivalent to our original possible worlds action in the sense that $W'w_1=W'w_2$ iff $Ww_1=Ww_2$ where $w_1$ and $w_2$ are sequences of sentences that never lead to a contradiction in the possible worlds sense. Additionally, the possible facts point of view is not as sensitive to contradiction: if someone says it's raining and someone says it's not raining, we don't lose any information we have about whether the cat is hungry.

The example above has shown that interest in the possible worlds approach to conversations and the operation of cylindrification in that context naturally leads to considering the possible facts approach and $\updown$-actions.

\section{Axiomatization and representation problems}\label{generalities}

The problems addressed in this paper fit into a general class of axiomatization and representation problems. Suppose we are interested in a class of structures $K$. Then we have an \emph{axiomatization problem}: Find a set of axioms $T$ (often of a desirable form) which characterizes the structures in $K$ up to isomorphism. Having selected a candidate set of axioms $T$, we are faced with a \emph{representation problem}: Show that every ``abstract'' model of $T$ is isomorphic to one of the ``concrete'' structures in $K$.

Familiar examples include Cayley's theorem, which says that every abstract group is isomorphic to a group of permutations of some set, and Stone's theorem, which says that every abstract Boolean algebra is isomorphic to an algebra of sets.

In this section we make some general observations about these problems, which will be useful in the special cases of $\updown$-actions and $\updown$-biactions. We assume that reader is familiar with the basic definitions of first-order logic (see \cite{Hodges}, for example).

\begin{definition}\label{elementary}
Let $\Sigma$ be a signature, and let $K$ be a class of $\Sigma$-structures.
\begin{itemize}
\item $K$ is \emph{elementary} is there is a first-order $\Sigma$-theory $T$ such that $K$ is the class of models of $T$. 
\item $K$ is \emph{pseudo-elementary} if there is a signature $\Sigma'\supseteq \Sigma$ and a first-order $\Sigma'$-theory $T'$ such that $K$ is the class of reducts to $\Sigma$ of models of $T'$.
\end{itemize}
\end{definition}

We will primarily consider classes with universal axiomatizations and Horn clause axiomatizations.

\begin{definition}\label{universal}
A \emph{universal sentence} is a sentence of the form $\A{\overline{x}}\vphi(\overline{x})$, where $\vphi$ is quantifier-free. A \emph{universal theory} is a set of universal sentences.
\end{definition}

\begin{definition}\label{horn}
A \emph{Horn clause} is a formula of the form 
\[
(\vphi_1\land \dots \land \vphi_n) \rightarrow \psi,
\]
where $\vphi_1,\dots,\vphi_n$ and $\psi$ are atomic. A \emph{Horn clause theory} is a set of Horn clauses. Identifying the Horn clause $\theta(\overline{x})$ with its universal closure $\A{\overline{x}}\theta(\overline{x})$, every Horn clause theory is a universal theory.
\end{definition}

Note that every atomic formula is a Horn clause, taking the left hand side of the implication to be the empty conjunction.

It is easy to check that every sentence which is equivalent to a universal sentence is preserved under substructure, and every sentence which is equivalent to a Horn clause is preserved under substructure and product. It is a well-known theorem of model theory that the converse statements are true.
 
\begin{theorem}[\cite{Hodges} Theorem 6.6.7 and Exercise 9.2.1]\label{preservation}
Let $K$ be a pseudo-elementary class of $\Sigma$-structures.
\begin{itemize}
\item $K$ is closed under substructure if and only if $K$ can be axiomatized by a universal theory in $\Sigma$.
\item $K$ is closed under substructure and product if and only if $K$ can be axiomatized by a Horn clause theory in $\Sigma$.
\end{itemize}
In particular, in either of these cases, $K$ is elementary.
\end{theorem}

The knowledge that a class $K$ is (pseudo-)elementary can be used to reduce the representation problem for $K$ to the case of finitely generated structures.

\begin{proposition}\label{fg} 
Let $K$ be a pseudo-elementary class which is closed under substructure, and let $T$ be a universal theory. If every finitely generated model of $T$ is in $K$, then every model of $T$ is in $K$.
\end{proposition}
\begin{proof}
By Theorem~\ref{preservation}, $K$ is elementary, axiomatized by a universal theory $T_K$. Given a model $A\models T$, we need to show that $A\models T_K$.

Let $\psi\in T_K$, written as $\A{\overline{x}}\vphi(\overline{x})$, with $\vphi$ quantifier-free, and let $\overline{a}$ be from $A$. Let $B_{\overline{a}}$ be the substructure of $A$ generated by $\overline{a}$. Then $B_{\overline{a}}\models T$, since $T$ is universal, and hence $B_{\overline{a}}\in K$, since it is finitely generated. Hence $B_{\overline{a}}\models\psi$, so $B_{\overline{a}}\models \vphi(\overline{a})$, and since $\vphi$ is quantifier-free, $A\models \vphi(\overline{a})$. 
\end{proof}

In the examples of Cayley's theorem and Stone's theorem, as well as in our cases of $\updown$-actions and $\updown$-biactions, the class $K$ is the class of substructures of some ``full'' structures. Then the representation problem becomes the problem of embedding each model of $T$ into one of these full structures.

When the full structures are obtained from sets by a construction which turns disjoint unions of sets into products of structures (e.g. in the case of Boolean algebras, but not in the case of groups), the class $K$ is controlled by the full structure on the one element set, in a way we will now make precise.

Fix a function $F$ associating to each set $X$ a structure $F(X)$, such that
\begin{enumerate}
\item If there is a bijection between $X$ and $Y$, then there is an isomorphism between $F(X)$ and $F(Y)$, and
\item $F$ turns disjoint unions of sets into products of structures, i.e.\ 
\[
F\left(\biguplus_{i\in I}X_i\right) \cong \prod_{i\in I}F(X_i).
\]
\end{enumerate}

Call the structures in the image of $F$ full, and let $K$ be the class of (structures isomorphic to) substructures of full structures.

\begin{proposition}\label{F(1)} 
Let $F$ and $K$ be as defined above, and let $1$ be the one element set $\{*\}$.
\begin{enumerate}
\item The class $K$ is closed under substructure and product.
\item Every structure $A\in K$ embeds canonically into a product of copies of $F(1)$, indexed by the set of homomorphisms from $A$ to $F(1)$.
\[
A \hookrightarrow \prod_{\text{Hom}_K(A,F(1))}F(1) \cong F(\text{Hom}_K(A,F(1)))
\]
\item If $K$ is pseudo-elementary, then it is elementary, axiomatized by the Horn clause theory of the structure $F(1)$.
\end{enumerate}
\end{proposition}
\begin{proof}
(1): $K$ is closed under substructure by definition. If $\{A_i\}_{i\in I}$ is a collection of structures in $K$, then each $A_i$ embeds in some full structure $F(X_i)$. Then $\prod_{i\in I}A_i$ embeds in $\prod_{I} F(X_i)\cong F(\biguplus_{I} X_i)$, so $\prod_{I}A_i$ is in $K$.

(2): First, observe that for all $X$, $X$ can be expressed as an $X$-indexed disjoint union of copies of $1$: $X = \biguplus_{x\in X}1$. So $F(X) \cong F(\biguplus_{x\in X}1) \cong \prod_{x\in X}F(1)$. Hence every structure $A$ in $K$ embeds into a product of copies of $F(1)$. 

For the canonical embedding, note that if $A$ embeds into some product of copies of $F(1)$, then for every pair of distinct elements $a$ and $b$ in the same sort of $A$, one of the coordinate maps $\vphi\colon A\to F(1)$ separates $a$ and $b$, i.e.\  $\vphi(a)\neq \vphi(b)$. Then if $A$ is in $K$, the map $A\to \prod_{\vphi\in \text{Hom}(A,F(1))}F(1)$ which is $\vphi$ on the component indexed by $\vphi$ is an embedding, since each of these separating maps appears in some coordinate.

(3): By Theorem~\ref{preservation}, any pseudo-elementary class closed under substructure and product is axiomatizable by a Horn clause theory. 

Let $\vphi$ be a Horn clause. If $\vphi$ is true in every structure in $K$, then clearly it is true of $F(1)$. Conversely, if $\vphi$ is true of $F(1)$, then since every $A$ in $K$ is isomorphic to a substructure of a product of copies of $F(1)$, and Horn clauses are preserved under substructures and products, $\vphi$ is true of $A$. 
\end{proof}

\begin{remark}\label{categories}
We have avoided the language of category theory above, as it is not necessary for our presentation, but it's worth observing how Proposition~\ref{F(1)} fits into a categorical framework. Let $\textsf{K}$ be the category whose objects are structures in $K$ and whose arrows are homomorphisms. Then the function $F$ can be extended to a functor $F\colon  \textsf{Set}^{\text{op}}\to \textsf{K}$, the functor $\text{Hom}_\textsf{K}(-,F(1))$ is left-adjoint to $F$, and the canonical embedding from Proposition~\ref{F(1)} is the unit map of this adjunction.
\end{remark}

\section{$\updown$-actions}\label{actions}

We begin by reviewing our notational conventions for actions. We view an action $(C,S)$ as an algebra in a two-sorted signature with a single function symbol $f\colon C\times S\to C$. When $c$ and $s$ are elements or variables of sorts $C$ and $S$, respectively, we write $cs$ for $f(c,s)$. We denote by $S^*$ the set of words in $S$. Given $c\in C$ and $w\in S^*$, $cw$ is an element of sort $C$. For all $w\in S^*$, let $f_w\colon C\to C$ be the function $c\mapsto cw$. We say $w$ is an \emph{identity operation} if $f_w$ is the identity function, and $w$ is a \emph{constant operation} (with value $d$) if $f_w$ is the constant function $f_w(c) = d$ for all $c\in C$.

Given a set $X$, we form an action $F(X)$ called the \emph{full $\updown$-action on $X$} by setting 
\begin{align*} 
C &= \{c\mid c\subseteq X\}\\ 
S &= \{(s^\down,s^\up)\mid s^\down,s^\up\subseteq X\text{ and }s^\up\subseteq  s^\down\}\\ 
f(c, (s^\down,s^\up)) &= (c \cap s^\down) \cup s^\up.
\end{align*}
An action is an \emph{$\updown$-action} if it is isomorphic to a subalgebra of the full $\updown$-action on some set $X$. In other words, an $\updown$-action is an action $(C,S)$ where each element of $C$ can be identified with a subset $c$ of some set $X$ and each element of $S$ can be identified with a pair $(s^\down,s^\up)$ of subsets of $X$ with $s^\up\subseteq s^\down$, such that the action of $s$ on $c$ is given by intersection with $s^\down$ and union with $s^\up$.

Note that the condition that $s^\up\subseteq s^\down$ implies that $(c\cap s^\down)\cup s^\up = (c\cup s^\up)\cap s^\down$, so the order of operations in the definition doesn't matter. This restriction is convenient but not important; in Proposition~\ref{other classes} below, we show that if we allow all pairs of subsets of $X$ in the $S$ sort, we get the same class of algebras up to isomorphism.

We will now apply the generalities of Section~\ref{generalities} to the case of $\updown$-actions.

\begin{proposition}
\label{pseudo-elementary actions}
The class of $\updown$-actions is pseudo-elementary.
\end{proposition}
\begin{proof}
Expand the signature by an additional sort $W$ and additional binary relations $\in\colon W\times C$, $\in^\down\colon W\times S$, and $\in^\up\colon W\times S$. Then let $T$ be the theory which asserts extensionality:
\[
\A{c,d\colon C} ((\A{w\colon W} w\in c\leftrightarrow w\in d)\rightarrow c=d)
\]
and
\[
\A{s,t\colon S}((\A{w\colon W} (w\in^\down s\leftrightarrow w\in^\down t) \wedge (w\in^\up s\leftrightarrow w\in^\up t))\rightarrow s=t),
\]
the subset condition on $S$:
\[
\A{s\colon S}(\A{w\colon W} (w\in^\up s\rightarrow w\in^\down s)),
\]
and the way $S$ acts on $C$:
\[\A{w\colon W}\A{c\colon C}\A{s\colon S} (w\in cs\leftrightarrow((w\in c\wedge w\in^\down s)\vee w\in^\up s)).
\]
Now, every $\updown$-action can clearly be expanded to become a model of $T$. Conversely, given a model of $T$, we may embed its reduct into the full $\updown$-action on $W$ by associating to $c\in C$ the set $\{w\in W\mid w\in c\}$ and to $s\in S$ the pair $(\{w\in W\mid w\in^\down s\},\{w\in W\mid w\in^\up s\})$. This is 1-1 by extensionality and is a homomorphism by the fourth sentence in $T$. So $T$ witnesses that the class of $\updown$-actions is pseudo-elementary.
\end{proof}

It is straightforward to verify that the operation $F$ which takes a set $X$ to the full $\updown$-action on $X$ turns disjoint unions of sets into products of algebras. Thus Proposition~\ref{F(1)} applies and we have:

\begin{corollary}\label{hornax}
The class of $\updown$-actions is axiomatized by the Horn clause theory of $F(1)$.
\end{corollary}

It's worth writing down $F(1)$ explicitly: $F(1) = (C(1),S(1))$, where, naming $0 = \emptyset$, we have $C(1) = \{0, 1\}$ and $S(1) = \{(1, 0), (0,0), (1,1)\}$. On $C$, $(1,0)$ acts as an identity operation, $(0,0)$ as a constant operation with value $0$, and $(1,1)$ as a constant operation with value $1$.

\begin{remark}
The canonical embedding described in Proposition~\ref{F(1)} takes on a particularly nice form for $\updown$-actions. Let $(C,S)$ be an $\updown$-action, and define $H = \text{Hom}((C,S),F(1))$. Then, by examining the composition
\[
(C,S) \hookrightarrow \prod_{H}F(1)\cong F\left(\biguplus_H 1\right) \cong F(H),
\]
we obtain the map
\begin{align*}
c &\mapsto \{f\in H\mid f(c) = 1\}\\
s &\mapsto (\{f\in H\mid f(s) = (1,0)\text{ or }f(s) = (1,1)\},\{f\in H\mid f(s) = (1,1)\}).
\end{align*}
\end{remark}

Now our goal is to characterize the class of $\updown$-actions by conditions which translate to Horn clause axioms.

Recall that an action is \emph{idempotent} if $css = cs$ for all $c\in C$ and $s\in S$. We say an action is \emph{previous redundant} if $csts = cts$ for all $c\in C$ and $s,t\in S$. Further, an action is \emph{fully previous redundant} if $csws = cws$ for $c\in C$, $s\in S$, and $w\in S^*$. Previous redundance is so called because from the point of view of the second $s$, the previous $s$ is redundant and can be removed.

\begin{lemma}\label{fullypr}
Any action which is idempotent and previous redundant is fully previous redundant.
\end{lemma}
\begin{proof}
By induction on the length of the word $w\in S^*$. The cases when $w$ has length $0$ and $1$ are covered by idempotence and previous redundance.

Now suppose that the length of $w$ is $n+1\geq 2$, and write $w$ as $w't$, where $w'$ is a word of length $n$. Then $csws = (csw')ts = (csw')sts$ by previous redundance. Applying the induction hypothesis to $csw's$, this is equal to $cw'sts = cw'ts = cws$, by another application of previous redundance.
\end{proof}

An \emph{$n$-step link} between $c$ and $d$ is a sequence $c = a_0$, $a_1$, $\dots$, $a_{n-1}$, $a_n = d$ of elements of $C$ and a sequence $w_1,\dots,w_n$ of words in $S^*$ such that for each $i = 1,\dots, n$, $a_{i-1}$ and $a_i$ are fixed points of $w_i$, i.e.\ $a_{i-1}w_i = a_{i-1}$ and $a_iw_i = a_i$. A \emph{strong link} between $c$ and $d$ is an $n$-step link, for some $n\geq 0$, such that additionally $cw_i = dw_i$ for all $i=1,\dots,n$. Every $c\in C$ is trivially strongly linked to itself (by a $0$-step link). A strong link between $c$ and $d$ is \emph{nontrivial} if $c\neq d$.

Note that there is a $1$-step link between any two elements $c$ and $d$, taking $w_1$ to be the empty word (or any identity operation). However, any nontrivial strong link must be at least two steps. Indeed, a $1$-step link between $c$ and $d$ is witnessed by $w\in S^*$ such that $cw = c$ and $dw = d$. But if this link is strong, then $c = cw = dw = d$. Similarly, no identity operation can appear in a nontrivial strong link.

The condition that all strong links are trivial is expressed by infinitely many Horn clauses, obtained by varying the natural number $n$ (the length of the $n$-step link) and the lengths of the sequences of variables $w_i$ of sort $S$ in the schema below.
\[
\Bigg(c = a_0 \wedge d=a_n\wedge \Bigg(\bigwedge_{i=1}^n cw_i=dw_i\wedge a_{i-1}w_i=a_{i-1}\wedge a_{i}w_i=a_{i}\Bigg)\Bigg)\rightarrow (c=d)
\]
We call this the strong links axiom. It is necessary to allow arbitrary words $w_i$ rather than just single elements of $S$ in this axiom, as is shown by Example~\ref{words are needed}.

We can now establish one half of our characterization.

\begin{proposition}\label{check}
The action $F(1)$ is idempotent, previous redundant, and has no nontrivial strong links. By Corollary~\ref{hornax}, these conditions are true in every $\updown$-action.
\end{proposition}
\begin{proof}
Idempotence is clear, since each element of $S$ acts as an identity or a constant operation on $C$. To check previous redundance, let $c\in C$, $s,t\in S$. If $s = (1,0)$, then $csts = ct = cts$, since $s$ acts as an identity on $C$. If $s = (0,0)$, then $csts = 0 = cts$, and if $s = (1,1)$, then $csts = 1 = cts$.

To check that all strong links are trivial,
we just need to see that $0$ and $1$ are not strongly linked in $F(1)$. If they were, then in particular there would be a $1$-step strong link between them, but we have already seen that all $1$-step strong links are trivial.
\end{proof}

Next, we pin down the equational theory of $\updown$-actions.

\begin{proposition}\label{pmequations}
The idempotent and previous redundant equations axiomatize the equational theory of $\updown$-actions.
\end{proposition}
\begin{proof}
That the $\updown$-actions are idempotent and previous redundant follows from Proposition~\ref{check}.

In the other direction, first note that the only terms in sort $S$ are single variables, and since there are $\updown$-actions in which $|S| > 1$, the only equation in sort $S$ which is universally true on $\updown$-actions is the tautology $s = s$.

So let $cs_1 \cdots s_n=dt_1\cdots t_m$ be some equation in sort $C$ that is universally true in $\updown$-actions. First we note that $c$ must be the same variable as $d$. Otherwise, in $F(1)$, put $c=0$, $d=1$, and put all $S$-variables equal to $(1,0)$. Then the two sides are different. 
 
By repeatedly applying idempotence and previous redundance on each side, we may assume that among the $s_i$ each variable occurs only once, and similarly for the $t_j$. 

Next, we observe that the two sides must have the same $S$-variables and hence the same length. Otherwise, without loss of generality, let $s_i$ be a variable that doesn't occur among the $t_j$. Again in $F(1)$, put $s_i=(0,0)$, put all other $S$-variables equal to $(1,0)$, and put $c=d=1$. Then the two sides are different.

So we are looking at an equation like $cs_1\cdots s_n=ct_1\cdots t_n$. We now show that $s_n=t_n$, then $s_{n-1}=t_{n-1}$, and so on down to $s_1=t_1$. 

If $s_n\neq t_n$, then we could put $s_n=(0,0)$ and $t_n=(1,1)$ and the two sides would be different. By induction, assume $s_i=t_i$ for $i>k$, and suppose for contradiction that $s_k\neq t_k$. We can put $s_i=t_i=(1,0)$ for $i>k$ and put $s_k=(0,0)$ and $t_k=(1,1)$. Then $cs_1\cdots s_n=0\neq 1 =ct_1\cdots t_m$.

Hence the equation $cs_1\cdots s_n = dt_1\cdots t_n$ is a tautology, from which the original equation follows by applications of idempotence and previous redundance.  
\end{proof}

Unlike actions expressed using set intersection (Theorem~\ref{intersection}), the class of $\updown$-actions does not have an equational axiomatization. This is demonstrated by the following example, which shows that the condition that all strong links are trivial does not follow from the equational theory.

\begin{example}
\label{example}
Let $C=\{c,d,e\}$, let $S=\{s,t\}$, and put $cs=ds = c$, $ct = dt=d$, and $es=et=e$. Letting $a_0=c$, $a_1=e$, $a_2=d$, and $w_1=s$, $w_2=t$ we get a $2$-step link between $c$ and $d$, and in fact this is a nontrivial strong link, since $cs=ds$ and $ct=dt$. Hence $(C,S)$ is not an $\updown$-action.

To see that this action is fully previous redundant, consider the equation $auwu = awu$ with $a\in C$, $u\in S$, and $w\in S^*$. If $a = e$, then both sides are $e$. Otherwise, both sides are $c$ or $d$ in accordance with whether $u$ is $t$ or $s$. 
\end{example}

In the proof of Theorem~\ref{main theorem}, we will use two auxiliary actions, $(C,S^*)$, and $(\overline{C},S)$, constructed from an action $(C,S)$.

Recall that $S^*$ is the set of words in $S$. Note that there is a natural action of $S^*$ on $C$, and that the action $(C,S)$ embeds into the action $(C,S^*)$.

\begin{lemma}\label{star} If $(C,S)$ is an idempotent and previous redundant action in which all strong links are trivial, then so is $(C,S^*)$.
\end{lemma}
\begin{proof}
Any word $w\in (S^*)^*$ is equivalent to a word $w'\in S^*$. Then any pair of elements in $C$ which are strongly linked in $(C,S^*)$ are also strongly linked in $(C,S)$, and hence all strong links are trivial in $(C,S^*)$.

For the other axioms, we show that $(C,S^*)$ is fully previous redundant. If $c\in C$, $w,x\in S^*$, then $cwxw = cxw$ by $n$ applications of full previous redundance in $(C,S)$, where $n$ is the length of the word $w$. 
\end{proof}

Define a binary relation $\mathbf{\sim}$ on $C$ by $c \sim d$ if and only if there exists $s\in S$ such that $s$ is not an identity operation and $c = cs$ and $d = ds$. When the action is idempotent, this is equivalent to putting $c\sim d$ when both $c$ and $d$ are in the image of a common non-identity operation. $\sim$ is a symmetric relation, so its reflexive and transitive closure $\approx$ is an equivalence relation. Explicitly, we have $c\approx d$ if and only if for some $n \geq 0$ there exist $a_0,\dots,a_{n}\in C$ and non-identity operations $s_1,\dots,s_n\in S$ such that $c=a_0$, $d=a_{n}$, $a_{i-1}s_i=a_{i-1}$, and $a_{i}s_i=a_{i}$ for $i=1,\ldots,n$. Let $\overline{C} = C/\approx$.

This definition is very similar to the definition of an $n$-step link, but here we require the witnesses $s_i$ to be in $S$, not $S^*$, and we exclude identity operations.  

\begin{lemma} 
\label{congruence actions}
For any fully previous redundant action $(C,S)$, $\approx$ is a congruence on $C$, i.e.\ $(\overline{C},S)$ inherits the structure of an action. Moreover, $(\overline{C},S)$ is an $\updown$-action.
\end{lemma}
\begin{proof}
We must check that for all $c,d\in C$ and $s\in S$, if $c\approx d$, then $cs \approx ds$. If $s$ is an identity operation, then $cs = c \approx d = ds$. If $s$ is not an identity operation, then in fact $cs\sim ds$ by idempotence. 

To show that $(\overline{C},S)$ is an $\updown$-action, we embed it in an $\updown$-action. Note that for all $s\in S$, $s$ is either an identity operation or a constant operation on $\overline{C}$. Indeed, if $s$ is an identity operation on $C$, then the same is true on $\overline{C}$. If $s$ is not an identity operation on $C$, then for all $a,b\in C$, $as \sim bs$ by idempotence, so $as = bs$ in $\overline{C}$, and $s$ is a constant operation on $\overline{C}$.

We define an embedding $\psi\colon (\overline{C},S)\to F(\overline{C}\biguplus S)$ as follows: 
\begin{align*}
c &\mapsto \{c\}\text{ if }c\in \overline{C}\\
s &\mapsto \begin{cases}(\overline{C}\cup\{s\},\emptyset) \text{ if } s\in S\text{ and } s \text{ is an identity operation}\\
(\{d,s\},\{d\}) \text{ if } s\in S\text{ and } s\text{ is a constant operation with value } d\end{cases}
\end{align*}

This map is clearly injective on $\overline{C}$, and the dummy element $s$ is included in $\psi(s)$ for all $s$ to ensure that it is injective on $S$.

Now if $c\in \overline{C}$ and $s\in S$ is an identity operation, then $\psi(c)\psi(s) = (\{c\}\cap (\overline{C}\cup\{s\}))\cup \emptyset = \{c\} = \psi(c) = \psi(cs)$. If $s\in S$ is a  constant operation with value $d$, then $\psi(c)\psi(s) = (\{c\}\cap \{d,s\}) \cup \{d\} = \{d\} = \psi(d) = \psi(cs)$.
\end{proof}

\begin{theorem}
\label{main theorem}
An action is an $\updown$-action if and only if it is idempotent and previous redundant and all strong links are trivial.
\end{theorem}
\begin{proof}
We established in Proposition~\ref{check} that all $\updown$-actions are idempotent and previous redundant and have no nontrivial strong links. It remains to show the converse.

By Propositions~\ref{pseudo-elementary actions} and~\ref{fg}, it suffices to consider finitely generated actions. But any finitely generated fully previous redundant action is actually finite, because any term in the generators is equivalent to one in which no generator appears more than once. We may thus proceed by induction on $|C|$.

Our plan is to embed $(C,S)$ into a product of $\updown$-actions, from which it follows by Proposition~\ref{F(1)} that it is an $\updown$-action. To do this, we observe that if, for every pair of distinct elements in the same sort of $(C,S)$, there is a homomorphism to some $\updown$-action separating these elements, then the product of all these maps is an injective map to the product of these $\updown$-actions.

To separate elements of the $S$ sort, define a map $\vphi\colon (C,S)\to F(S)$ by $c \mapsto \emptyset$ for all $c\in C$ and $s \mapsto (\{s\},\emptyset)$ for all $s\in S$. Then for all $c\in C$ and $s\in S$, $\vphi(c)\vphi(s) = \emptyset = \vphi(cs)$, so $\vphi$ is a homomorphism, and $\vphi$ is injective on $S$. 

In the base case, when $|C|=1$, the map described above is injective on all of $(C,S)$, and we're done. So let $|C|>1$ and let $c\neq d$ in $C$ be two elements to separate.

\emph{Case 1:} There exists $t\in S$ such that $ct \neq dt$, and $t$ is not an identity operation.

We define a map $\vphi\colon (C,S)\rightarrow (C,S^*)$ by $c\mapsto ct$ for $c\in C$ and $s\mapsto st$ for $s\in S$. This is a homomorphism, since for all  $c\in C$ and $s\in S$, $\vphi(c)\vphi(s) = ctst = cst = \vphi(cs)$ by previous redundance. Since $ct \neq dt$, $\vphi(c) \neq \vphi(d)$. 

By Lemma~\ref{star}, $(C,S^*)$ is a previous redundant action in which all strong links are trivial, and the image of $\vphi$ is a subalgebra $(Ct,St)\subseteq (C,S^*)$, so the same is true of $(Ct,St)$.

We will show that $|Ct| < |C|$. Then we will be done with this case since by induction $(Ct,St)$ will be an $\updown$-action. By definition $Ct\subseteq C$. Suppose for contradiction it were all of $C$. Then for all $c\in C$, $c = dt$ for some $d\in C$, so $ct = dtt = dt = c$, and $t$ is an identity operation on $C$, contradiction.

\emph{Case 2:} For all $t\in S$, either $ct = dt$, or $t$ is an identity operation.

By Lemma~\ref{congruence actions}, the quotient map $q\colon (C,S)\to(\overline{C},S)$ is a homomorphism to an $\updown$-action. We'll be done if we show that $q$ separates $c$ and $d$, i.e.\ that $c\not\approx d$.

Suppose for contradiction that $c\approx d$. This is witnessed by sequences $c = a_0,a_1,\dots,a_n = d$ in $C$ and $s_1,\dots,s_n$ in $S$ such that for all $i$, $a_{i-1}s_i = a_{i-1}$, $a_is_i = a_i$, and $s_i$ is not an identity operation. But then $cs_i = ds_i$, so this data would also witness  that $c$ and $d$ are strongly linked, contradicting the assumption that $(C,S)$ has no nontrivial strong links.
\end{proof}

We conclude this section by considering the question of what changes if, in the definition of the full $\updown$-action, the requirement that $s^\up\subseteq s^\down$ is dropped. Formally, we have a new construction $F'$ of actions from sets, defined by $F'(X) = (C',S')$ where 
\begin{align*}
C' &= \{c\mid c\subseteq X\}\\
S' &= \{(s^\down,s^\up)\mid s^\down,s^\up\subseteq X\}\\
f(c,(s^\down,s^\up)) &= (c\cap s^\down) \cup s^\up.
\end{align*}

Say an action is an $\updown'$-action if it is isomorphic to a subalgebra of $F'(X)$ for some set $X$. It is easy to check once again that the class of $\updown'$-actions is pseudo-elementary and that $F'$ turns disjoint unions of sets into products of algebras, so Proposition~\ref{F(1)} applies.

Intuitively, if an element $x$ is in $s^\up$, it doesn't matter whether it is in $s^\down$: if intersection with $s^\down$ removes it, it will just get added in again by union with $s^\up$. So in moving from $F(X)$ to $F'(X)$, we haven't made a substantial change; we have only added some extra elements of the $S$ sort of $F'(X)$ which have the same action on $C$ as elements that were already in $F(X)$. The following proposition makes this precise.

\begin{proposition}\label{other classes}
Every $\updown'$-action is an $\updown$-action, and vice versa.
\end{proposition}
\begin{proof}
By Proposition~\ref{F(1)}, the $\updown$-actions and $\updown'$-actions are the classes of structures generated under product and substructure by $F(1)$ and $F'(1)$, respectively, so it suffices to show that $F(1)$ is an $\updown'$-action and $F'(1)$ is an $\updown$-action. 

We have $F(1) = (C,S)$ and $F'(1) = (C',S')$, where
\begin{align*} 
C &= C' = \{0,1\}\\
S &= \{(1,0), (0,0), (1,1)\}\\
S' &= \{(1,0), (0,0), (1,1), (0,1)\}
\end{align*}

Now clearly $F(1)$ is an $\updown'$-action, since it embeds in $F'(1)$. In the other direction, since $(0,1)$ and $(1,1)$ act on $C$ in the same way, we can embed $F'(1)$ into an $\updown$-action in a way that separates them with a dummy element $x$. Define a map $F'(1)\to F(1\cup \{x\})$ which is the identity on $C$ and acts as follows on S':
\begin{align*}
(1,0)&\mapsto (1,0)\\
(0,0)&\mapsto (0,0)\\
(1,1)&\mapsto (1,1)\\
(0,1)&\mapsto (1\cup\{x\},1).\qedhere
\end{align*}
\end{proof}

\section{$\updown$-biactions}
\label{biactions}

A \emph{biaction} $(C,S^\down,S^\up)$ is a pair of functions $f\colon C\times S^\down\to C$ and $g\colon C\times S^\up\to C$. We write $f(c,s)$ as $cs$ and $g(c,t)$ as $ct$.

Given a set $X$, we form a biaction $F(X)$ called the \emph{full $\updown$-biaction on $X$} by setting $C=S^\down=S^\up=\mathcal{P}(X)$, and for $c\in C$, $s\in S^\down$, and $t\in S^\up$, we put $cs=c\cap s$ and $ct=c\cup t$.

A biaction is an \emph{$\updown$-biaction} if it is isomorphic to a subalgebra of the full $\updown$-biaction on some set $X$. In other words, an $\updown$-biaction is a biaction where the elements of $C$, $S^\down$, and $S^\up$ can be identified with sets in such a way that $cs=c\cap s$ when $s\in S^\down$ and $ct=c\cup t$ when $t\in S^\up$. 

Every $\updown$-biaction gives rise to an $\updown$-action by combining $S^\down$ and $S^\up$ into one sort. Formally, if $(C,S^\down,S^\up)$ is a subalgebra of the full $\updown$-biaction on $X$, we can identify the element $s\in S^\down$ with $(s,\emptyset)$ and $t\in S^\up$ with $(X,t)$ in the full $\updown$-action on $X$. However, we can not in general go the other direction. That is, given an $\updown$-action $(C,S)$ we can not in general divide $S$ into two parts $S^\down$ and $S^\up$ so as to have an $\updown$-biaction (see Example~\ref{separating sorts}). In this sense there are more $\updown$-actions than $\updown$-biactions.

We now present axioms for $\updown$-biactions. First we note that $\updown$-biactions are commutative in both the $S^\down$ and $S^\up$ sorts in the sense that $cst=cts$ whenever $s$ and $t$ are both in $S^\down$ or both in $S^\up$. This is obvious from the definition of $\updown$-biactions because intersection and union are associative and commutative. Of course, elements of $S^\down$ do not commute with elements of $S^\up$ in general.

Next we note that $\updown$-biactions are idempotent and previous redundant. That is, for all $c\in C$, $s\in S^\down$ and $t\in S^\up$, we have $css = cs$, $ctt = ct$, $csts = cts$, and $ctst = cst$. This is because the action obtained by combining $S^\down$ and $S^\up$ into one sort is an $\updown$-action, and we've already observed that $\updown$-actions are idempotent and previous redundant. 

We have only stated previous redundance for variables $s$ and $t$ of different sorts. This is because if $s$ and $t$ are in the same sort, $csts = cts$ follows from commutativity and idempotence. Just as in Lemma~\ref{fullypr}, idempotence and previous redundance are enough to imply full previous redundance: for all $c\in C$, $s\in (S^\down\cup S^\up)$, and $w\in (S^\down\cup S^\up)^*$, $csws = cws$.

We have already introduced enough axioms to describe the equational theory of $\updown$-biactions.

\begin{proposition}\label{biequations}
The equations expressing idempotence, previous redundance, and commutativity in the sorts $S^\down$ and $S^\up$ axiomatize the equational theory of $\updown$-biactions.
\end{proposition}
\begin{proof}
Similar to the proof of Proposition~\ref{pmequations}.
\end{proof}

With the equational theory under our belt, we may now more easily present an example of an $\updown$-action which can't be reinterpreted as an $\updown$-biaction.

\begin{example}\label{separating sorts}
Let $C=\{c,d,e\}$ where $c=\{1\}$, $d=\{2\}$, and $e=\{3\}$. Let $S=\{s_c,s_d,s_e\}$ where $s_c=(\{1\},\{1\})$, $s_d=(\{2\},\{2\})$, and $s_e=(\{3\},\{3\})$. Each $s_x$ acts as the constant function with value $x$. Clearly $(C,S)$ is an $\updown$-action. The question under consideration is whether we can divide $S$ into two parts $S^\up$ and $S^\down$ so that the resulting biaction is an $\updown$-biaction. Any way of doing this will involve putting two elements of $S$ into the same sort, say $s_x$ and $s_y$ (where $x\neq y$). If we indeed have an $\updown$-biaction we should have $y=cs_xs_y=cs_ys_x=x$, which is a contradiction.
\end{example}

Once again, the equational theory is not enough to axiomatize the class in question, as the following example illustrates.

\begin{example}
Let $C=\{c,d\}$, $S^\down=\{s\}$, and $S^\up=\{t\}$. Define $cs=ct=ds=dt=d$. This biaction is idempotent, previous redundant, and commutative in $S^\down$ and $S^\up$. However, it can't be an $\updown$-biaction because we can't go from $c$ to $d$ by removing elements by $s$ on the one hand and adding elements by $t$ on the other ($cs=d$ implies $d\subseteq c$ and $ct=d$ implies $c\subseteq d$).
\end{example}

So we need to add Horn clause axioms to supplement our equational ones. The first axiom is called the \emph{basic subset axiom}. Let $s\in S^\down$ and $t\in S^\up$. If $cs=ds$ and $ct=dt$ and $es=et$, then $c=d$. Let's see why this axiom is true for the $\updown$-biactions. First note that $es=et$ implies that $t\subseteq et = es \subseteq s$. Next, $cs=ds$ implies $c$ and $d$ agree inside $s$, and $ct=dt$ implies $c$ and $d$ agree outside $t$. Since $t$ is a subset of $s$, we get that $c$ and $d$ agree everywhere. 

Next we add a series of modified versions of the basic subset axiom. For each word $w$ consisting of variables of sorts $S^\down$ and $S^\up$, we add $w$ to the end of each term that occurs in the basic subset axiom to form a new axiom. That is, we get an axiom
\[csw=dsw\wedge ctw=dtw\wedge esw=etw\rightarrow cw=dw\]
for each word $w$. Let's call all these axioms the \emph{extra subset axioms}.

Let's check that the extra subset axioms are true in $\updown$-biactions. We will do this by induction as follows. Suppose that we have a Horn clause 
\[
(\star)\hspace{.25in} (x_1=y_1\wedge\cdots\wedge x_n=y_n)\rightarrow (x_{n+1}=y_{n+1}),
\]
where $x_i,y_i$ are terms, which is universally true in $\updown$-biactions. Let $s$ be a variable of sort $S^\down$ or $S^\up$. We wish to show that
\[
(x_1s=y_1s\wedge\cdots\wedge x_ns=y_ns)\rightarrow (x_{n+1}s=y_{n+1}s)
\]
is also universally true in $\updown$-biactions. This will be enough, since each extra subset axiom can be built up from the basic subset axiom adding one variable at a time. Consider an $\updown$-biaction $(C,S^\down,S^\up)$ and assignment of variables so that $x_is=y_is$ for all $1 \leq i \leq n$. 

In the case $s$ is of sort $S^\down$, form a new $\updown$-biaction $(C\cap s,S^\down\cap s,S^\up\cap s)$ which is the restriction of the original $\updown$-biaction to $s$. In detail, $C\cap s=\{c\cap s\mid c\in C\}$, $S^\down\cap s=\{t\cap s\mid t\in S^\down\}$, and $S^\up\cap s=\{t\cap s\mid t\in S^\up\}$. There is an obvious homomorphism $\vphi$ from the original to the restriction given by intersection by $s$ on each sort. Since $\vphi(x_1)=x_1s$ and $\vphi(y_1)=y_1s$ and so on, we have by assumption 
\[
\vphi(x_1)= \vphi(y_1)\wedge\cdots\wedge\vphi(x_n)= \vphi(y_n)
\]
and, since $\vphi$ is a homomorphism, these equations are an instance of the premises of $(\star)$ in the restriction. Since the restriction is an $\updown$-biaction, we get the conclusion of $(\star)$, $\vphi(x_{n+1})=\vphi(y_{n+1})$. Hence $x_{n+1}s=y_{n+1}s$ in $(C,S^\down,S^\up)$, as desired.

In the case $s$ is of sort $S^\up$, we form a new $\updown$-biaction $(C\cup s,S^\down\cup s,S^\up\cup s)$, which is essentially the restriction of the original biaction to the complement of $s$. The argument goes just as in the $S^\down$ case. Alternatively, this case follows from the duality of $\updown$-biactions. 

\begin{example} \label{words are needed} We provide an example showing that the extra subset axioms do not follow from previous redundance, commutativity in $S^\down$ and $S^\up$, and the basic subset axiom. This example also shows that a weakened version of the strong links axiom (for $\updown$-actions) is not sufficient. In detail, the weakened strong links axiom is as follows: When $a_0,a_1,\ldots,a_n\in C$ and $s_1,\ldots,s_n\in S$ with $a_0s_i=a_ns_i$ and $a_{i-1}s_i=a_{i-1}$ and $a_is_i=a_i$ for $i=1,\ldots,n$, then $a_0=a_n$. In the actual strong links axiom we allow $s_i\in S$ to be replaced by an arbitrary word $w_i\in S^*$. 

Consider the biaction given by the diagram below (as usual, it may also be thought of as an action). There are three sentences $s$, $t$, and $u$. We put $s\in S^\down$, and $t,u\in S^\up$. The six elements $\{c,d,e,f,1,2\}$ of $C$ form two components. If a sentence fixes an element, then that arrow is not shown.
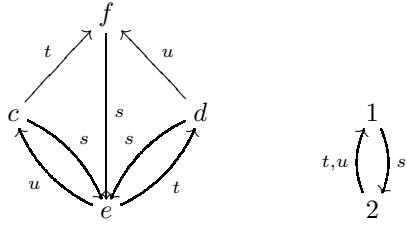
\[
\xymatrix{
 & f \ar[dd]^s & & & \\
 c \ar@/^/[dr]^s \ar[ur]^t & & d \ar@/_/[dl]_s \ar[ul]_u & &  1\ar@/^/[d]^s \\
 & e \ar@/^/[ul]^u \ar@/_/[ur]_t & & & 2 \ar@/^/[u]^{t,u}
}
\]
Each component separately is an $\updown$-biaction. To realize the first component, $\{c,d,e,f\}$, as an $\updown$-biaction, we may let $c=\{\#\}$, $d=\{\$\}$, $e=\emptyset$, $f=\{\#,\$\}$, $s=\emptyset$, $t=\{\$\}$, and $u=\{\#\}$. To realize the second component, $\{1,2\}$, as an $\updown$-biaction, we may let $1=\{\%\}$, $2=\emptyset$, $s=\emptyset$, $t=\{\%\}$, and $u=\{\%\}$. Since each of the equational axioms has the same $C$-variable occurring on both sides, the fact that each component separately is an $\updown$-biaction implies that the biaction as a whole satisfies the equational axioms. This biaction also satisfies the basic subset axiom: for no $a\in C$ do we have $as=at$ or $as=au$. The fact that it satisfies the basic subset axiom actually implies that the associated action satisfies the weakened version of the strong links axiom, though this can also be separately checked. Also, the associated action does not satisfy the actual strong links axiom. To see this, note that $csu=c$, $1su=1$, $1st=1$, and $dst=d$, so there is a link between $c$ and $d$, but this link is also strong because $csu=c=dsu$ and $cst=d=dst$. The fact that the associated action is not an $\updown$-action implies that the biaction can't satisfy the extra subset axioms. For a specific example, note that $csu=dsu$, $ctu=dtu$, and $1su=1tu$, yet $cu=c\neq f=du$. 
\end{example}

\begin{lemma}
\label{adding t lemma}
Let $B=(C, S^\down, S^\up)$ be a biaction satisfying the axioms (idempotence, previous redundance, commutativity in $S^\down$ and $S^\up$, and the basic and extra subset axioms). Let $t$ be an element of $S^\down$ or $S^\up$. Define a biaction $B_t=(C_t,S^\down_t,S^\up_t)$ by putting
\begin{align*}
C_t&=\{ct\in C\mid c\in C\}\\
S^\down_t&=S^\down\\ 
S^\up_t&=S^\up,
\end{align*}
and defining the action as follows: given $s\in S^\down_t$ or $s\in S^\up_t$, and $c\in C_t$ we define the $B_t$-action of $s$ on $c$ to be $cst$. Then $B_t$ also satisfies the axioms. 
\end{lemma}
\begin{proof}
The equational axioms are easy to check. For example, let $c\in C_t$, and $s,u\in S^\down_t$. Then commutativity of $s$ and $u$ on $c$ in $B_t$ amounts to the equation $cstut = cutst$ in $B$, which is equivalent to $csut = cust$ in $B$ by previous redundance, and this last equation is true by commutativity in $B$.

It remains to check $B_t$ satisfies the subset axioms. Let one of the subset axioms be given, written as
\[
(\star) \hspace{.25in} (x_1=y_1\wedge\cdots\wedge x_n=y_n)\rightarrow (x_{n+1}=y_{n+1}).
\]
Suppose variables are given assignments in such a way that $x_1=y_1,\ldots,x_n=y_n$ in $B_t$. Then, applying previous redundance to remove intermediate $t$'s, we have $x_it=y_it$ in $B$ for all $1\leq i \leq n$. And so we may cite the extra subset axiom obtained from $(\star)$ by adding $t$ to every term to conclude that $x_{n+1}t=y_{n+1}t$ in $B$, i.e.\ $x_{n+1}=y_{n+1}$ in $B_t$.
\end{proof}

\begin{theorem}
\label{biaction theorem}
A biaction is an $\updown$-biaction if and only if it is idempotent, previous redundant, and commutative in $S^\down$ and $S^\up$, and it satisfies the basic and extra subset axioms.
\end{theorem}

\begin{proof}
When introducing the axioms, we proved that $\updown$-biactions satisfy these axioms. So it remains to show that a biaction satisfying these axioms is an $\updown$-biaction. 

We can do the same tricks we did in the case of $\updown$-actions: $\updown$-biactions form a pseudo-elementary class by essentially the same argument as in Proposition~\ref{pseudo-elementary actions} for $\updown$-actions.  Also, $\updown$-biactions are the subalgebras of full $\updown$-biactions of the form $F(X)$, and $F$ is a function which turns disjoint unions of sets into products of algebras, so the class of $\updown$-biactions is closed under substructure and product (Proposition~\ref{F(1)}). So given a biaction $B=(C,S^\down,S^\up)$ satisfying the axioms, it suffices to find, for each pair of distinct elements in the same sort, a homomorphism to an $\updown$-biaction separating these two elements.

Our axioms are universal (in fact they are Horn clauses) and so by Proposition~\ref{fg}, we need only check that every finitely generated model of the axioms is an $\updown$-biaction. But, once again, full previous redundance implies that every finitely generated model is finite, and we can do induction on $|C|$. 

First we show that no matter what the size of $C$, we can separate elements in sort $S^\down$ and in sort $S^\up$. Let's consider $S^\down$. Define a map $\vphi\colon (C,S^\down,S^\up)\to F(S^\down)$ by $c\mapsto \emptyset$ for all $c\in C$, $s\mapsto \{s\}$ for all $s\in S^\down$, and $t\mapsto\emptyset$ for all $t\in S^\up$. It's easy to check that $\vphi$ is a homomorphism and it is injective on $S^\down$.  The sort $S^\up$ works dually.

We turn now to $C$. In the base case, when $|C|=1$, there is no pair of distinct elements to separate in sort $C$, and so we're done. So let $|C|>1$ and let $c\neq d$ in $C$. 

\emph{Case 1:} There is a $t\in S^\down\cup S^\up$ such that $ct\neq dt$ and $t$ is not an identity operation. Consider the biaction $B_t=(C_t,S^\down_t, S^\up_t)$ defined as in Lemma~\ref{adding t lemma}. 

We claim that $|C_t|<|C|$. Of course $|C_t|\leq |C|$, since $C_t\subseteq C$. If $|C_t|=|C|$, then for every $c\in C$ there is $d\in C$ such that $dt=c$. Then $ct=dtt=dt=c$ by idempotence, and so $t$ is an identity operation, contrary to assumption.

Because $|C_t|<|C|$ and $B_t$ satisfies the axioms, by the inductive hypothesis we can conclude that $B_t$ is an $\updown$-biaction.

Consider the map $\vphi\colon B\to B_t$ defined by $\vphi(c)=ct$ for $c\in C$ and $\vphi(s)=s$ for $s$ in either $S^\down$ or $S^\up$. This is a homomorphism, since $\vphi(cs) = cst = ctst = \vphi(c)\vphi(s)$, and it has $\vphi(c)\neq\vphi(d)$ by assumption. Hence we've found a separating homomorphism to an $\updown$-biaction.

\emph{Case 2:} For every $t\in S^\down\cup S^\up$, either $ct=dt$ or $t$ is an identity operation.

We form a quotient $\bar{B}=(\bar{C},S^\down,S^\up)$ of $B=(C,S^\down,S^\up)$ as follows. For $a,b\in C$, we put $a\approx b$ when $as=a$ and $bt=b$ for some non-identity operations $s,t\in S^\down$ or some non-identity operations $s,t\in S^\up$, or $a=b$. In other words, we identify all the elements of $C$ which are fixed by any non-identity operation in $S^\down$, and similarly we identify all the elements of $C$ which are fixed by any non-identity operation in $S^\up$. To show this is transitive, it suffices to show that there is no element which is fixed by both a non-identity operation in $S^\down$ and a non-identity operation in $S^\up$. Suppose that for some $e\in C$, $es=e=et$ where $s\in S^\down$ and $t\in S^\up$ and $s$ and $t$ are not identity operations. Then also $cs=ds$ and $ct=dt$ by assumption, and so the premises of the basic subset axiom are satisfied. We conclude that $c = d$, which is a contradiction.

Now let's check that $\approx$ is a congruence. If $a\approx b$ and $s$ is in $S^\down$ or $S^\up$, then $as\approx bs$ because either $s$ is an identity operation and $as=a\approx b = bs$, or $ass=as$ and $bss=bs$ witness that $as\approx bs$. 

Next we show $c\not\approx d$. Suppose for contradiction that $c\approx d$, and suppose that this is witnessed by $s,t\in S^\down$ non-identity operations such that $cs=c$ and $dt=d$ (the case $s,t\in S^\up$ is the same). By our assumptions, we get $c=cs=ds$ and $d = dt = ct$. But then $c = ds = cts = cst = ct = d$, a contradiction.

So the quotient map is a homomorphism from $B$ to $\bar{B}$ that separates $c$ and $d$. It remains to show that $\bar{B}$ is an $\updown$-biaction. Observe that every non-identity operation $s\in S^\down$ is a constant operation with the same constant in each case, since if $a,b\in C$ and $s,t\in S^\down$ are non-identity operations, $ass = as$ and $btt = bt$ witnesses that $as \approx bt$. The same is true for $S^\up$. The argument for transitivity above showed also that these constants must be different.  Let's call them $a_\down$ and $a_\up$. We define a map $\vphi\colon \bar{B}\to F(\bar{C}\cup S^\down\cup S^\up)$ as follows. 
\begin{align*}
a&\mapsto\begin{cases}
\{a\}\cup S^\up &\text{ if }a\neq a_\down,a_\up\\
S^\up &\text{ if }a=a_\down\\
\bar{C}\cup S^\up &\text{ if }a=a_\up
\end{cases}\\
s\in S^\down&\mapsto\begin{cases}
\bar{C}\cup\{s\}\cup S^\up&\text{ if }s\text{ is an identity}\\
\{s\}\cup S^\up &\text{ if }s\text{ is constant}
\end{cases}\\
s\in S^\up&\mapsto\begin{cases}
\{s\}&\text{ if }s\text{ is an identity}\\
\bar{C}\cup\{s\}&\text{ if }s\text{ is constant}
\end{cases}
\end{align*}
It is easily checked that this is a homomorphism and it is 1-1 on each sort. 
\end{proof}
 
\section{Set bands}\label{set bands}

Let $(C,S)$ be a full $\updown$-action on some set. What product $\cdot$ can we put on $S$ so that $(cs)t=c(s\cdot t)$? The following calculation gives an answer. Let $s=(s^\down, s^\up)$ and $t=(t^\down,t^\up)$. Then,
\begin{align*}
(cs)t &= (((c\cap s^\down)\cup s^\up) \cap t^\down)\cup t^\up\\
&= (c\cap s^\down \cap t^\down) \cup (s^\up\cap t^\down)\cup t^\up\\
&=(c\cap (s^\down \cap t^\down)) \cup (c\cap t^\up) \cup (s^\up\cap t^\down)\cup t^\up\\
&= (c \cap ((s^\down\cap t^\down)\cup t^\up)) \cup ((s^\up\cap t^\down)\cup t^\up).
\end{align*}

This motivates the following definition. Given a set $X$, we form an algebra $(F(X),\cdot)$, called the \emph{full set band on $X$}, by setting
\begin{align*}
F(X)&=\{(s^\down,s^\up)\mid s^\down,s^\up\subseteq X\text{ and } s^\up\subseteq s^\down\}\\
(s^\down,s^\up)\cdot(t^\down,t^\up)&=((s^\down\cap t^\down)\cup t^\up,(s^\up\cap t^\down)\cup t^\up)
\end{align*}
In general, an algebra is called a \emph{set band} if it isomorphic to a subalgebra of the full set band on some set $X$.  

Set bands are indeed bands (idempotent semigroups), and their definition involves intersection and union, hence the name ``set bands". Further, set bands are right regular ($xyx=yx$). One way to check this is to observe that every set band is the semigroup of operations for some $\updown$-action $(C,S)$, i.e.\ the semigroup of functions on $C$ generated by $\{f_s\colon C\to C\mid s\in S\}$, and right regularity follows from previous redundance. That conversely every semigroup of operations of an $\updown$-action (which is, a priori, a quotient of a set band) is a set band follows from Theorem~\ref{set bands theorem}.

We state without proof (due to the similarity with $\updown$-actions) a few facts about set bands.

\begin{proposition}\label{equations}
Associativity, idempotence, and right regularity axiomatize the equational theory of set bands.
\end{proposition}

The class of set bands is pseudo-elementary, and $F$ turns disjoint unions into products, and so the set bands are the subalgebras of the products of the algebra $F(1)$ and admit a Horn clause axiomatization. 

In \cite{Margolis}, Margolis et al.\ study the class of subsemigroups of the ``hyperplane face monoids", which they identify as the quasivariety of algebras generated (under subalgebra and product) by a certain three-element algebra. This algebra is essentially $F(1)$, with the superficial difference that the order of multiplication is reversed (e.g. it is left regular instead of right regular), and so their algebras are exactly the set bands, after reversing the multiplication. They show that this quasivariety is axiomatized by associativity, idempotence, left regularity, and a schema of Horn clauses which is called (CC) in \cite{Margolis}, and which is very similar to our condition on $\updown$-actions that all strong links are trivial. Their method led directly to the proof of Theorem~\ref{main theorem}, and inspired the proof of Theorem~\ref{biaction theorem}. We think it is interesting that the same class of algebras arose in these two ways, with such different motivations.

For completeness, we'll state a version of the theorem characterizing set bands, adapted to our vocabulary. We say that two elements $c,d$ of an algebra $(S,\cdot)$ are \emph{strongly linked} when for some natural number $n$ there exist $a_0,\dots,a_n$ and $s_1,\dots,s_n$ in $S$ such that $c=a_0$, $d=a_n$, and for $i = 1,\dots,n$, $a_{i-1}s_i=a_{i-1}$, $a_{i}s_i=a_{i}$, and $cs_i=ds_i$, and we say that the strong link between $c$ and $d$ is \emph{trivial} when $c = d$. 

\begin{theorem}\label{set bands theorem}
Set bands are axiomatized by associativity, idempotence, right regularity, and the condition that all strong links are trivial.
\end{theorem}

\end{document}